\documentclass[12pt]{amsart}
\usepackage{graphicx}

\usepackage{hyperref}
\hypersetup{
	colorlinks,
	citecolor=black,
	filecolor=black,
	linkcolor=black,
	urlcolor=black
}

\usepackage{tikz-cd} 
\usepackage{hyperref}           
\usepackage{calc}               
\usepackage{enumerate}          
\usepackage{mathrsfs}             
\usepackage{pxfonts}            
\usepackage{verbatim}           
\usepackage{fancyhdr}           
\usepackage{url}                
\usepackage{gloss}              
\usepackage[marginratio=1:1,height=8.5in,width=6.5in,tmargin=1.25in]{geometry}
\usepackage{comment}            
\usepackage[normalem]{ulem}     
\usepackage{todonotes}          

\makeatletter
\def\@tocline#1#2#3#4#5#6#7{\relax
  \ifnum #1>\c@tocdepth 
  \else
    \par \addpenalty\@secpenalty\addvspace{#2}%
    \begingroup \hyphenpenalty\@M
    \@ifempty{#4}{%
      \@tempdima\csname r@tocindent\number#1\endcsname\relax
    }{%
      \@tempdima#4\relax
    }
    \parindent\z@ \leftskip#3\relax \advance\leftskip\@tempdima\relax
    \rightskip\@pnumwidth plus4em \parfillskip-\@pnumwidth
    #5\leavevmode\hskip-\@tempdima
      \ifcase #1
       \or\or \hskip 1em \or \hskip 2em \else \hskip 3em \fi%
      #6\nobreak\relax
    \dotfill\hbox to\@pnumwidth{\@tocpagenum{#7}}\par
    \nobreak
    \endgroup
  \fi}
\makeatother

\newcommand{\mb}{\mathbb}

\newcommand{\mf}{\mathfrak}

\newcommand{\on}{\operatorname}

\providecommand{\to}{\longrightarrow }

\newtheorem*{theorem*}{Theorem}
\newtheorem*{proposition*}{Proposition}
\newtheorem*{conj}{Question}
\newtheorem{lemma}{Lemma}

\begin{document}

\author[Neal Coleman]{Neal Coleman}

\title[Polya's eigenvalue conjecture is false for spheres]{Polya's eigenvalue conjecture is false for spheres}

\address{} 
\email{coleman.neal@gmail.com} 

\begin{abstract}
By comparing the Laplace spectrum of the sphere $\mathbb{S}^n$ to its Weyl function $w(x) = \frac{\omega_n}{(2\pi)^n}|\mb{S}^n|x^{n/2}$, we show that no analogue of Polya's eigenvalue conjecture holds in general for Riemannian manifolds with positive sectional curvature.
\end{abstract}
\maketitle


\thispagestyle{empty} 


\section*{Introduction}
Let $M$ be a compact Riemannian manifold of dimension $n$, possibly with boundary. Denote by $|M|$ the volume of $M$ with respect to the measure induced by the Riemannian metric.

Consider the Dirichlet eigenvalue problem for the Laplace operator $-\on{div}\on{grad}$. Its spectrum consists of finite multiplicity eigenvalues bounded below and accumulating to infinity. Denoted the eigenvalues by $\lambda_1\leq\lambda_2\leq\cdots$. Let $N(x) = \# \{ k\ |\ \lambda_k\leq x\}$ be the eigenvalue counting function.

Weyl's law describes the asymptotic growth of the counting function:
\[ N(x) = \frac{\omega_n}{(2\pi)^n}|M|x^{n/2} + R(x) = w(x) + R(x)\]
where $\omega_n$ is the volume of the unit ball in $\mb{R}^n$ and $R$ is a remainder term of order $o(x^{n/2})$. In what follows, for convenience we call $w(x) = \frac{\omega_n}{(2\pi)^n}|M|x^{n/2}$ the Weyl function of $M$.

The result was proved first by Weyl in 1911 \cite{Weyl1911}, and sharpened and generalized through the twentieth century by many authors using a variety of techniques. For more information see a survey by Ivrii \cite{Ivrii2016}.

P\'olya's conjecture states that for domains in $\mb{R}^n$,
\[ N < w \]
P\'olya originally conjectured this in 1954 \cite{Polya1954} and proved it in 1961 for domains that tile $\mb{R}^2$ \cite{Polya1961}; the argument generalizes immediately to any dimension.

This conjecture remains open in full generality. Advances have been made by Li and Yau \cite{LiYau1983}, Urakawa \cite{Urakawa1984}, and others. Recently, Levitin, Polterovich, and Sher have proven that this conjecture is true for the unit disk in $\mb{R}^2$ \cite{LevitinPolterovichSher2022}. This is the first non-tiling domain for which P\'olya's conjecture is known to hold.

McKean and Singer \cite{McKeanSinger1967} used a heat trace argument to show that for $M$ a closed manifold, Weyl's law satisfies
\[ N(x) = w(x) + \frac{\omega_{n-2}}{6(2\pi)^{n-1}} \bigg( \int_M K \bigg)\ x^{n/2 - 1} + R(x) \]
where $K$ is the scalar curvature of $M$ and $R$ is of order $o(x^{n/2 - 2})$.\footnote{McKean and Singer prove this for the heat trace. The heat trace is the Laplace transform of the eigenvalue counting function. The version here immediately follows from applying Feller's version of the Hardy-Littlewood Tauberian theorem.}

Based on this expansion of the counting function, one may ask whether some form of P\'olya's conjecture holds for closed Riemannian manifolds with nonzero scalar curvature:
\begin{conj}
If $M$ has positive (resp. negative) scalar curvature, is it the case that
\[ N > (\mbox{resp.} <)\ w ?\]
\end{conj}

We provide a counterexample:
\begin{theorem*}
Fix $n>0$. Let $N$ be the eigenvalue counting function of the round sphere $\mb{S}^n$ and $w$ its Weyl function. Then $N$ crosses $w$ infinitely many times and so Polya's eigenvalue conjecture does not hold.
\end{theorem*}

\begin{figure}[h]
	\includegraphics[scale=0.75]{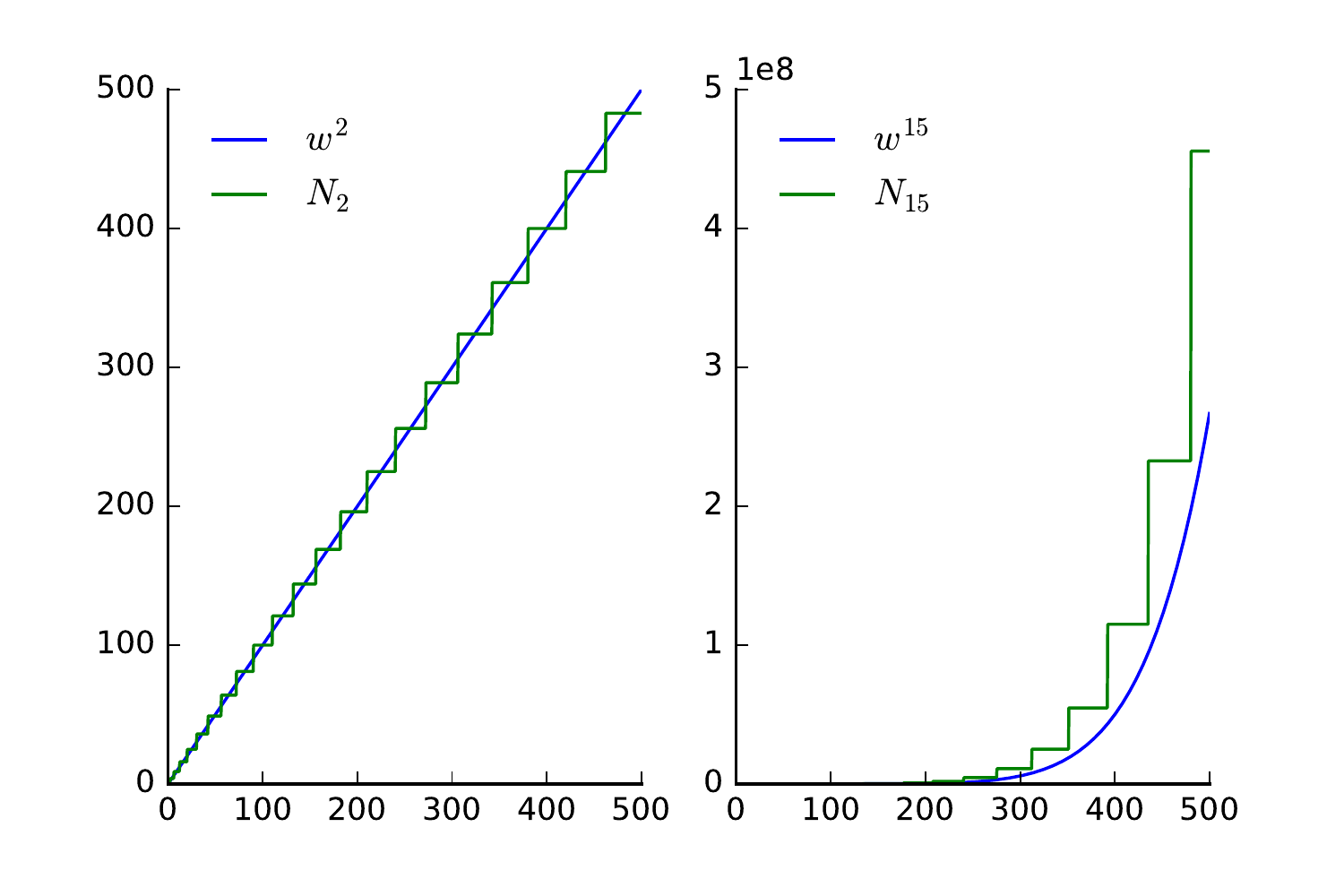}
	\caption[Weyl function and counting functions for $\mb{S}^2$ and $\mb{S}^{15}$]{We illustrate the Theorem by plotting the eigenvalue counting function and Weyl function for $\mb{S}^2$ (left) and $\mb{S}^{15}$ (right). (Note the vertical axis of the right plot is in units of $10^8$.) Figure first appeared in author's dissertation\cite{Coleman2017}.}\label{sphere_fig}
\end{figure}

\section{Laplace spectrum for round spheres}

Recall the following facts about round spheres $\mb{S}^n = \{x\in\mb{R}^{n+1}\ |\ |x|=1\}$ and their Laplace spectra. The Riemannian metric on $\mb{S}^n$ is induced by the Euclidean metric on $\mb{R}^{n+1}$. The eigenfunctions of the Laplacian on $\mb{S}^n$ are linear combinations of restrictions to the sphere of homogeneous harmonic polynomials on $\mb{R}^{n+1}$.

The distinct eigenvalues are $v_k = k(k+n-1)$. Their multiplicity\footnote{The dimension of homogeneous polynomials of degree $k$ is $\binom{k+n}{k}$ and the Laplace operator maps degree-$k$ homogeneous polynomials onto degree-$(k-2)$ homogeneous polynomials.} is 
\[
\mbox{multiplicity of}\ v_k = \begin{cases}
1, &k=0\\
n+1, &k=1\\
\binom{n+k}{k} - \binom{n+k-2}{k-2}, &k\geq 2
\end{cases}
\]

We prove two lemmas simplifying the sphere's counting function and its Weyl function.

\begin{lemma}[Eigenvalue counting function of $\mb{S}^n$]
The counting function of the sphere $\mb{S}^n$ is $N(x) = N\big( \{v_k\ |\ v_k \leq x \} \big)$. Adopting the convention that $\binom{m}{j} = 0$ for $j<0$, we have
\begin{align*} 
N(v_k) &= \binom{n+k}{k} + \binom{n+k-1}{k-1} \\
 	&= \frac{2}{n!}\bigg(k+\frac{n}{2}\bigg)(k+n-1)\cdots(k+1)
\end{align*}
\end{lemma}
\begin{proof}
The value of the counting function at a point $x$ is equal to the sum of the multiplicities of the eigenvalues less than or equal to $x$. The sum telescopes:
\begin{align*}
N_n(k(k+n-1)) &= \sum_{j\leq k} \binom{n+j}{j} - \binom{n+j-2}{j-2} \\
    &= \binom{n+k}{k} + \binom{n+k-1}{k-1}
\end{align*}

Expanding yields:
\begin{align*}
\binom{k+n}{k} + \binom{k+n-1}{k-1} &= \frac{1}{n!}(k+n)(k+n-1)\cdots(k+1)\\ &\ \ \ \ + \frac{1}{n!}(k+n-1)(k+n-2)\cdots k \\
	&= \frac{1}{n!}(k+n-1)\cdots (k+1)(k+n+k) \\
	&= \frac{2}{n!}\bigg(k+\frac{n}{2}\bigg)(k+n-1)\cdots(k+1)
\end{align*}
as claimed.
\end{proof}

\begin{lemma}[Sphere Weyl function]
The Weyl function of $\mb{S}^n$ is
\[ w(x) = \frac{2}{n!}x^\frac{n}{2} \]
\end{lemma}
\begin{proof}
Recall $\omega_n$ denotes the volume of the unit ball in $\mb{R}^n$. Denote by $\mf{s}_n$ the Riemannian volume of $\mb{S}^n$. Note that $\omega_0 = 1$ and $\mf{s}_0 = 2$. 

The following recursive relation is established by integrating in spherical and toroidal coordinates, respectively:
\[ \begin{cases}  
\omega_n &= \frac{1}{n}\omega_{n-1} \\
\mf{s}_n &= 2\pi\omega_{n-1} \\
\end{cases} \]
By the recursive relation we observe that $\omega_n\mf{s}_n = \frac{2\pi}{n}\omega_{n-1}\mf{s}_{n-1}$. Inductively we have 
\[\omega_n\mf{s}_n = \frac{(2\pi)^n}{n!}\omega_0\mf{s}_0 = 2\frac{(2\pi)^n}{n!}. \]

As the sphere's Weyl function is
\[ w(x) = \frac{\omega_n}{(2\pi)^n}|\mb{S}^n|x^{n/2} = \frac{\omega_n\mf{s}_n}{(2\pi)^n}|\mb{S}^n|x^{n/2} \]
substitution in the numerator yields the claim.
\end{proof}

\section{Proof of theorem}

We establish the Theorem by proving the following
\begin{proposition*}
For all $k$, \[ w(v_k) < N(v_k). \]
For sufficiently large $k$,
\[ N(v_k) < w(v_{k+1}) \]
\end{proposition*}

The first statement in this Proposition was proven in the author's dissertation\cite{Coleman2017}.

Because $N$ is constant on the interval $(v_k, v_{k+1})$, the second inequality implies that 
\[ \lim_{\epsilon\to 0} N(v_{k+1}-\epsilon) < w(v_{k+1}) \]
for sufficiently large $k$.

As $w$ and $N$ are both continuous on the interval $(v_k, v_{k+1})$, they must have crossed: the Theorem follows from the proposition combined with the intermediate value theorem.

\begin{proof}
We first prove that for all $k$, we have $w(v_k) < N(v_k)$. By Lemmas 1 and 2, this is equivalent to:
\[\frac{2}{n!}\big(k(k+n-1)\big)^{n/2} < \frac{2}{n!}\bigg(k+\frac{n}{2}\bigg)(k+n-1)\cdots(k+1).\]

Canceling the factor of $2/n!$ and squaring both sides, the assertion holds if and only if
\[ \bigg[k(k+n-1)\bigg]^n < \bigg(k+\frac{n}{2}\bigg)^2\big(k+n-1\big)^2\cdots (k+1)^2 \]
As each factor is nonzero, this is true if and only if
\[ \frac{k(k+n-1)}{(k+n/2)^2}\prod_{j=1}^{n-1} \frac{k(k+n-1)}{(k+j)(k+n-j)} < 1. \]
The result follows from observing that $j(n-j) \leq \frac{n^2}{4}$ for $j=1,2,\ldots,n-1$ and expanding numerators and denominators:
\[ k^2 + (n-1)k < k^2 + nk + j(n-j) \leq k^2 + nk + \frac{n^2}{4} \]
Each factor in the product is less than one. Therefore the entire product is less than one, establishing the inequality.

We now prove the second inequality. Again applying Lemmas 1 and 2, the assertion that $w(v_{k+1}) > N(v_k)$ is equivalent to
\[ \left[ (k+1)(k+n-1) \right]^n > \bigg(k+\frac{n}{2}\bigg)^2\big(k+n-1\big)^2\cdots (k+1)^2 \]

Analogous to above, this holds if and only if
\[ \frac{(k+1)(k+n)}{(k+n/2)^2}\prod_{j=1}^{n-1} \frac{(k+1)(k+n)}{(k+j)(k+n-j)} > 1 \]

Suppose $k > \frac{n^2}{4} - n.$ Then for all $j=1,\ldots, n-1$, we have
\[ k^2 + (n+1)k + n > k^2 + nk + \frac{n^2}{4} \geq k^2 + (n+1)k + j(n-j) \]
Each factor in the product is no less than $1$, thus establishing the desired inequality.
\end{proof}

This result is not sharp. For sufficiently large $n$, numerical experiments indicate a intermediate regime where $N > w$. See the example of $\mb{S}^{15}$ in Figure \ref{sphere_fig}.


\providecommand{\bysame}{\leavevmode\hbox to3em{\hrulefill}\thinspace}
\providecommand{\MR}{\relax\ifhmode\unskip\space\fi MR }
\providecommand{\MRhref}[2]{%
  \href{http://www.ams.org/mathscinet-getitem?mr=#1}{#2}
} 
\providecommand{\href}[2]{#2}

\end{document}